\documentclass[12pt]{amsart}
\usepackage{fullpage,url,amssymb,enumerate,colonequals}
\usepackage{mathrsfs}
\usepackage[section]{placeins}
\usepackage{MnSymbol}
\usepackage{extarrows}
\usepackage{lscape}
\usepackage[all,cmtip]{xy}
\usepackage{graphicx} 
\usepackage{todonotes}
\usepackage[indLines=true, noEnd=false]{algpseudocodex}
\usepackage[OT2,T1]{fontenc}
\usepackage{color}
\usepackage[
        colorlinks, citecolor=darkgreen,
        backref,
        pdfauthor={Nuno Freitas},
]{hyperref}
\usepackage{comment}
\usepackage{multirow}

\numberwithin{equation}{section}

\newtheorem{lemma}[equation]{Lemma}
\newtheorem{theorem}[equation]{Theorem}
\newtheorem{proposition}[equation]{Proposition}

\newtheorem{corollary}[equation]{Corollary}

\newtheorem{claim*}{Claim}

\theoremstyle{definition}
\newtheorem{remark}[equation]{Remark}

\definecolor{darkgreen}{rgb}{0,0.5,0}
\definecolor{rem}{rgb}{0.8,0,0}
\definecolor{new}{rgb}{0.3,0.1,0.9}
\definecolor{reply}{rgb}{0,0,0.8}
\definecolor{gray}{gray}{0.7}

\renewcommand{\det}{\text{det}}

\newcommand{\F}{\mathbb{F}}
\newcommand{\PP}{\mathbb{P}}
\newcommand{\Q}{\mathbb{Q}}

\newcommand{\Z}{\mathbb{Z}}
\newcommand{\rhobar}{{\overline{\rho}}}

\newcommand{\calO}{\mathcal{O}}
\newcommand{\fp}{\mathfrak{p}}

\newcommand{\Kbar}{{\overline{K}}}
\newcommand{\Qbar}{{\overline{\Q}}}
\newcommand{\Gal}{\text{Gal}}
\newcommand{\Frob}{\text{Frob}}

\newcommand{\PSL}{\text{PSL}}
\newcommand{\GL}{\text{GL}}

\renewcommand{\ker}{\text{ker}\:}

\DeclareMathOperator{\Tr}{\text{Tr}}

\DeclareMathOperator{\Sel}{Sel}

\DeclareMathOperator{\Sp}{Sp}

\newcommand{\fq}{\mathfrak{q}}
\newcommand{\cM}{\mathcal{M}}
\newcommand{\cT}{\mathcal{T}}

\newcommand{\cQ}{\mathcal{Q}}

\def\Gal{\operatorname{Gal}}
\def\GL{\operatorname{GL}}
\def\Sel{\operatorname{Sel}}

\def\Frob{\operatorname{Frob}}

\def\QQ{\mathbb{Q}}
\def\O{\mathcal{O}}
\def\F{\mathbb{F}}

\title{Comparing Galois representations in the residually reducible case}

\date{\today}

\author{Nuno Freitas}
\address{Instituto de Ciencias Matem\'aticas (ICMAT),
          Nicol\'as Cabrera 13-15
         28049 Madrid, Spain}
\email{nuno.freitas@icmat.es}

\author{Ignasi S\'anchez-Rodr\'iguez}
\address{Universitat de Barcelona (UB),
        Gran Via de les Corts Catalanes 585, 
        08007 Barcelona, Spain}
\email{ignasi.sanchez@ub.edu}

\thanks{Freitas was partly supported by the PID 2022-136944NB-I00 grant of the MICINN (Spain)}
\thanks{S\'anchez-Rodr\'iguez was party supported by the FPI grant of the PID PID2019-107297GB-I00 grant of the MICINN (Spain)}

\subjclass[2010]{}

\begin{document}

\begin{abstract}
Let $n \geq 2$ and $p$ be a prime. Let $K$ be a number field and consider two Galois representations $\rho_1, \rho_2 : \Gal(\Kbar / K) \to \GL_n(\Z_p)$ having residual image a $p$-group. We explain and implement an algorithm that makes effective a result of Loïc Greni\'e to decide wether the semisimplifications of $\rho_1$ and~$\rho_2$ are isomorphic.

As an application, we show that an irreducible representation $\rho : G_{\Q(\sqrt{-3})} \to \GL_2(\Z_3)$ unramified outside 3 is determined by the characteristic polynomials of Frobenius elements at five primes of small norm. As an additional check, we apply it to a 2-adic example studied by Greni\'e, recovering Greni\'e's result in a fully automated way.
\end{abstract}

\maketitle

\section{Introduction}

Let $p$ be a prime and $E/\Q_p$ a finite extension.
Let $K \subset \Qbar$ be a number field with absolute Galois group $G_K := \Gal(\Qbar / K)$ and $S$ a finite set of primes in~$K$.
It follows from the foundational work of Faltings~\cite{Faltings} that two continuous representations $\rho_1, \rho_2 : G_K \to \GL_n(E)$ unramified outside $S$ are isomorphic if and only if they have the same traces at a finite set~$\cT$ of Frobenius elements. In view of this, Serre proposed a method to make~$\cT$ computationally accessible for representations valued in $\GL_2(\Q_2)$ with absolutely irreducible residual image. This method, now known as {\it the Faltings--Serre method}, has seen various applications over the years, for example, to the modularity of elliptic curves over imaginary quadratic fields~\cite{DGP}. In the more recent work~\cite{paramodularity}, the authors extended the method to the representations valued in $\Sp_2(\Z_2)$ with absolutely irreducible residual image and applied it to prove the paramodularity of an abelian surface of conductor 277; moreover, in the PhD thesis of Mattia Sanna~\cite{mattia} the method was extended to $\GL_2(\Q_3)$ in most cases of residually irreducible image.

In the case of residually reducible image, which is the focus of this work, the Faltings--Serre method was first extended by Livn\'e~\cite{livne} to the case of 2-adic 2-dimensional representations with residual image a 2-group. Then, in~2007, Loïc Greni\'e~\cite{Grenie} generalized Livn\'e's work
to any~$p$-adic $n$-dimensional representation still assuming that the residual image is a $p$-group (see Theorem~\ref{thm:Grenie}). As an application, he established the equivalence of two $3$-dimensional non-selfdual $2$-adic representations studied by Van Geemen and Top~\cite{vanGeemenTop}, one arising on the cohomology of a variety and the other an automorphic rerpesentation for $\GL_{3,\Q}$.


A major limitation of the Faltings--Serre--Livn\'e--Greni\'e methods is the requirement to construct large number fields even when~$n$, $p$ and~$S$ are small. To minimize this problem one can try to use additional information specific to the representations being studied. For example, in the work of Greni\'e mentioned above, he explains how to
use facts about groups of order~128 to avoid a very large computation. Another
beautiful example of this is the work of Duan~\cite{duan}, where he proves
that a representation constructed by van Geemen and Top (see~\cite[\S4]{vanGeemenTop}) is isomorphic to a quadratic twist of the symmetric square of the Tate module of an elliptic curve. In this case, for the computations to become feasible, Duan develops a variant of Greni\'e's method leveraging the fact that the representations involved are selfdual of dimension~3 and that the Galois group of the maximal pro-2 extension of~$\Q(\sqrt{-3})$ unramified outside $2$ is a free pro-2 group with 2~generators.

In this work, we discuss a variant of Greni\'e's theorem (see Theorem~\ref{thm:smaller}) and implement an algorithm to apply
it in full generality. Being designed to work in general, the implementation does not take into account any specific information that may reduce the required computations. This has the advantage
that, when it works, the algorithm produces a set of Frobenius elements, depending only on $K$, $E$, $n$, $p$ and~$S$, that applies to any two continuous representations $\rho_i : G_K \to \GL_n(E)$ with residual image a $p$-group, regardless of their origin. Therefore, for each choice of parameters, we only need to run the algorithm once, potentially enabling the creation of a database. In particular, we derive the following.

Let $K=\Q(\sqrt{-3})$ and $\fq_3 \mid 3$
be the unique prime in $K$ above 3. Note that $2$ is inert in~$K$,
let $\fp_7^1$ and $\fp_7^2$ be the two primes in $K$ above 7
and let $\fp_{q}^{\ast}$ be any of the primes in $K$ above~$q=19,73$.
\begin{corollary}\label{cor:3adic}
Let $\rho_1, \rho_2 : G_K \to \GL_2(\Z_3)$ be continuous representations unramified outside~$\fq_3$ having the same determinant. Suppose their determinant is trivial modulo~$3$ and consider the list of primes
 \[\cT = \{2\calO_K, \mathfrak{p}_7^1, \mathfrak{p}_7^2, \mathfrak{p}_{19}^{\ast}, \mathfrak{p}_{73}^{\ast}\}.\]
Then $\rho_1$ and~$\rho_2$ have isomorphic semisimplifications if and only if
$\rho_1(\Frob_t)$ and $\rho_2(\Frob_t)$ have the same trace for all $t \in T$.
\end{corollary}
To illustrate this result, we apply it to study modularity of the abelian surfaces with good reduction away from~3 in the LMFDB~\cite{LMFDB}; their modularity also follows from other results but, to our knowledge, this is the first successful application of a $3$-adic Faltings--Serre type method in the residually reducible case.

As an additional example and sanity check, we also apply our algorithm to the example studied by Greni\'e. As expected, we recover his result in a fully automated fashion.

\begin{corollary}\label{cor:2adic}
Let $n=3$ or $n=4$ and $\rho_1, \rho_2 : G_\Q \to \GL_n(\Z_2)$ be continuous representations unramified outside 2 with residual image a 2-group. Consider the sets of primes
\[\cT = \{5,7,11,17,23,31\}
  \quad \text{ and } \quad \cT' = \{5,7,11,17,19,23,31,73,137,257,337\}.\]
Then $\rho_1$ and~$\rho_2$ have isomorphic semisimplifications if and only if $\rho_1(\Frob_t)$ and $\rho_2(\Frob_t)$ have the same characteristic polynomials for all $t \in \cT$ or the same traces for all $t \in \cT'$.
\end{corollary}
To close this introduce, we remark that this result was used by Greni\'e
for $n=3$ and is also crucial for $n=4$ in the recent work of Visser~\cite{visser} studying abelian surfaces over~$\Q$ with good reduction away from 2 and full 2-torsion over~$\Q$.

\subsection{Computational resources}
For the computations needed in this paper, we used {\tt Magma} computer algebra system \cite{magma}, version V2.28-20. All our code is available at \cite{git} and the repository instructions explain how it is used.

\subsection{Acknowledgements}
We thank John Cremona and Mart\'in Az\'on for helpful discussions. We thank John Voight and Stephan Elsenhans for their help with Magma and specifically the Galois group computations.
Part of this work was completed while both authors were at the Max Planck Institute for Mathematics in Bonn. The authors are grateful to the Max Planck Institute for Mathematics in Bonn for its hospitality and financial support.

\section{A theorem of Greni\'e}
\label{sec:grenie}

Here we recall the main theorem of~\cite{Grenie}.
Let $n \geq 2$ and $p$ be a prime. Let $K$ be a number field and $S$ a finite set of primes in~$K$. Let $E/\Q_p$ be a finite extension with ring of integers~$\calO_E$ and maximal ideal $\mathfrak{p}_E$. Let also
 \[\rho_1,\; \rho_2 \; : \; G_K \longrightarrow \GL_n(E)\]
be two continuous representations unramified away from $S$. We will also need the extension~$K_S$ as defined in~\cite[\S 2.1]{Grenie}. Take $K_0 = K$ and define by induction $K_{i+1}$ to be the maximal abelian extension of~$K_i$ unramified outside $S_i$ and having Galois group isomorphic to a finite sum of copies of $\F_p$,
where~$S_i$ is the set of primes in $K_i$ above all the primes in~$S$.
Let $\varepsilon = 0$ if $p \ne 2$ and $\varepsilon = 1$ if $p=2$.
Let $N = n[E : \Q_p]$ and $r = N^{2(1+\varepsilon)}\frac{N(N-1)}{2}$.
Let $\lambda$ and~$m$ be the smallest integers satisfying $2^\lambda \geq r$ and $p^m \geq n$, respectively. Finally, define $K_S = K_{\lambda + \varepsilon + m}$.

The maximality of the layers implies that~$K_S / K$ is Galois (see also Proposition~\ref{prop:correspondence} for a more general statement).
We can now state the main result in~\cite{Grenie}.

\begin{theorem}[Greni\'e]\label{thm:Grenie} Let $T$ be a finite set of primes in $K$ disjoint from $S$ such that each maximal cyclic subgroup of $\Gal(K_S/K)$ has a generator of the form
$\Frob(\mathfrak{t}/t)$ where $t \in T$ and $\mathfrak{t} \mid t$ in $K_S$.
 Assume also that
 \begin{itemize}
     \item[$(i)$] $\forall g\in G_K$, $\operatorname{CharPoly}(\rho_1(g)) \equiv \operatorname{CharPoly}(\rho_2(g)) \equiv (X-1)^n \pmod{\mathfrak{p}_E}$; and
     \item[$(ii)$] $\forall t \in T$, $\operatorname{CharPoly}(\rho_1(\Frob_t)) = \operatorname{CharPoly}(\rho_2(\Frob_t))$ where $\Frob_t \in G_K$ is any Frobenius element at~$t$.
\end{itemize}
Then $\rho_1$ and $\rho_2$ have isomorphic semisimplifications.
\end{theorem}

\begin{remark}The statement of \cite[Theorem 3]{Grenie} is more general than the above. However, some of the hypothesis in it are there only to guarantee that we can reduce the problem to the case stated above; see Remark 4 and first part of the proof of Theorem 3 in {\it loc. cit.}.
\end{remark}

\begin{remark}
The field $K_S$ is often much larger than necessary. Indeed, suppose $E=\Q_2$;
for $n=2$ we have $K_S = K_6$, but~\cite[Proposition~4.7]{livne} shows that it suffcies to take $K_S = K_1$. For $n=3$ and $n=4$
we obtain $K_S = K_{10}$ and $K_S = K_{12}$, respectively, but the proof of Corollary~\ref{cor:2adic} shows that $K_S = K_3$ in both cases. (We will see this happens because $m=2$ in both cases.)
\end{remark}

\section{Auxiliary results}
\label{sec:auxiliary}
Let $p$ be a prime and $K$ be a number field containing the $p$-th roots of unity.
Let $S_p$ be the set of primes in $K$ above $p$ and $S \supset S_p$ a finite set of primes
in~$K$.

The correctness of the algorithm described in Section~\ref{sec:steps} depends on various elementary results which we introduce in this section; we refer to~\cite[\S 5.2]{cohen} for comprehensive discussions of computational aspects of $p$-Selmer groups.

Let $M/L$ and $L/K$ be finite extensions with $L/K$ Galois.
Let $S_L$ denote the set of primes in $L$ dividing the primes in~$S$. To ease notation, we will say that $M/L$ is unramified outside $S$ if it is unramified outside $S_L$. The \(p\)-Selmer group attached to~$S_L$ is the finite abelian group
\[
\Sel_p(L,S) := \left\{ a \in L^*/(L^*)^p : v_{\mathfrak{p}}(a) \equiv 0 \pmod{p} \quad \text{for all } \mathfrak{p} \notin S_L \right\}.
\]
The Selmer group has exponent~$p$ and so
it can be thought of as $\F_p$-vector space of finite dimension.
More precisely, after choosing $\{\alpha_1,\ldots,\alpha_d\}$ a minimal set of generators for~$\Sel_p(L,S)$,
we can identify it with a $d$-dimensional $\F_p$-vector space~$V$,
where the additive structure of~$V$ is converted into the multiplicative structure of
$\Sel_p(L,S)$ via the bijection $$V \ni  (e_1,\ldots, e_d) \mapsto
\alpha_1^{e_1} \cdot \ldots \cdot \alpha_d^{e_d} \in \Sel_p(L,S).$$

From now on, we will interchange between both sides of this map without further mention.

\begin{remark} When creating the $p$-Selmer group attached to~$S_L$, the software {\tt Magma} automatically creates the above space~$V$ and returns it together with the map to $\Sel_p(L,S)$.
\end{remark}

The group $G = \Gal(L/K)$ acts on $\Sel_p(L,S)$ via $\sigma(a \cdot (L^*)^p) := \sigma(a) \cdot (L^*)^p$
for all $\sigma \in G$.
A $\F_p[G]$-submodule of $\Sel_p(L,S)$ is a $\F_p$-subspace $N \subset \Sel_p(L,S)$ that is invariant for the action of~$G$. We will refer to them simply as $G$-submodule of $\Sel_p(L,S)$.

From Kummer theory, the field $L(\sqrt[p]{\alpha_1},\dots, \sqrt[p]{\alpha_d})$ where $\alpha_i$ are generators of $\Sel_p(L,S)$
is the maximal $p$-elementary extension of~$L$ unramified outside $S$ and, moreover, maximality implies that it is Galois over~$K$; in fact, we have the following more general correspondence.
\begin{proposition} \label{prop:correspondence} We have a
1-1 correspondence between $G$-submodules of $\Sel_p(L, S)$ and
$p$-elementary extension of~$L$ unramified outside $S$ such that $F/K$ is Galois. More precisely,

1) If $N = \langle \alpha_1,\dots, \alpha_r \rangle \subset \Sel_p(L, S)$ is a $G$-submodule, then
$L(\sqrt[p]{\alpha_1},\dots, \sqrt[p]{\alpha_r}) / K$ is Galois.

2) Conversely, if $F/L$ be a $p$-elementary extension unramified outside $S$ such that $F/K$ is Galois, then there are $\alpha_1,\ldots,\alpha_r \in \Sel_p(L, S)$ generating a $G$-submodule
$N \subset \Sel_p(L, S)$ such that $F = L(\sqrt[p]{\alpha_1},\dots, \sqrt[p]{\alpha_r})$.
\end{proposition}
\begin{proof}
1) Let $N = \langle \alpha_1,\dots, \alpha_r \rangle \subset \Sel_p(L, S)$ be
a $G$-submodule and $L_N = L(\sqrt[p]{\alpha_1},\dots, \sqrt[p]{\alpha_r})$.
Clearly, $L_N / K$ is Galois if and only if the Galois closure of $L_j = L(\sqrt[p]{\alpha_j})$ over $K$ is contained in $L_N$ for all~$j$. Let $\sigma \in G_K$.
The elements $\beta_j = \sqrt[p]{\alpha_j}$ and $\sigma(\beta_j)$ are roots of the polynomials
$x^p + \alpha_j$ and $x^p + \sigma(\alpha_j)$ both in $L[x]$ (the latter because $L/K$ is Galois). Since $\sigma(\alpha_j) \in N$ by assumption, we have $\sigma(\beta_j) \in L_N$, that is the Galois closure $L_j$ over~$K$ is contained in $L_N$.

2) Let $F/L$ be as in the statement. Define $N \subset \Sel_p(L,S)$ be the $\F_p$-subspace  generated by the elements $\alpha \in \Sel_p(L,S)$ such that $x^p - \alpha$ has a root (hence all roots) in $F$.
Let $\sigma \in G$. Since $F/K$ is Galois, there is $\tilde{\sigma} \in \Gal(F/K)$ such that $\tilde{\sigma}|_L = \sigma$ and $\beta = \tilde{\sigma}(\sqrt[p]{\alpha}) \in F$. The element~$\beta$ is a root of $x^p - \sigma(\alpha)$ in $L[x]$, hence $\alpha \in N$. Thus $N$ is a $G$-submodule of $\Sel_p(L,S)$.
\end{proof}

\begin{lemma}\label{lem:nminus1}
Let $N \subset \Sel_p(L, S)$ be a $G$-submodule of dimension~$d\geq 1$.
Then there is a $G$-submodule $M \subset N$ of dimension~$d-1$.
In particular, $N$ admits a 1-dimensional $G$-submodule.
\end{lemma}

\begin{proof} 
Let $N = \langle \alpha_1,\dots, \alpha_d \rangle \subset \Sel_p(L, S)$ be a $G$-submodule and $ M = L(\sqrt[p]{\alpha_1},\dots, \sqrt[p]{\alpha_d})$.

We have $X:= \Gal(M/L) \simeq C_p^d$ because $M/L$ is $p$-elementary.

By Proposition~\ref{prop:correspondence}, we know that
$M/K$ is Galois and the statement is equivalent to the existence of a field $L \subset M' \subset M$ such that $[M : M'] = p$ and $M'/K$ is Galois. By the Galois correspondence, this is equivalent to the existence of a subgroup $H \subset X$ of order~$p$ that is normal in $\Gal(M/K)$.

Recall that $L/K$ is Galois. Therefore, $X$ is normal in $\Gal(M/K)$ and so $\Gal(M/K)$ acts on~$X$ by conjugation.
It follows from \cite[\S8.3, Lemma 3]{Serre} that
$$\# X \equiv \# X^{\Gal(M/K)} \equiv 0 \pmod{p},$$
hence there is $h \in X$ such that $h \neq 1$ and $g h g^{-1} = h$ for all $g \in \Gal(M/K)$.
We take $H$ to be generated by~$h$.
The last statement follows by successive applications of the first part.
\end{proof}

\begin{lemma}\label{lem:removeExtension}
Let $N\subseteq N'$ be two $G$-submodules of $\Sel_p(L, S)$ with corresponding
extensions $F / L$ and $F'  / L$ given by Proposition~\ref{prop:correspondence}, respectively.
Let $\fq$ be a prime in~$K$.
Suppose there is a prime $\cQ \mid \fq$ in $F$ satisfying $f(\cQ \mid \fq) > p^k$.
Then $f(\cQ' \mid \fq) > p^k$ for all $\cQ' \mid \fq$ in $F'$.
\end{lemma}

\begin{proof}
    This is clear since $F$ is a subextension of $F'$ by Proposition~\ref{prop:correspondence}.
\end{proof}

\begin{lemma}\label{lem:resDegree}
    Let $L/K$ be a $p$-elementary extension. Let $\mathfrak{q}$ be a prime of $K$ unramified in~$L$ and $\cQ \mid \fq$ a prime in~$L$. Then, the residual degree $f(\cQ/\fq)$ divides~$p$.
\end{lemma}

\begin{proof} The residual degree $f(\cQ/\fq)$ coincides with the order of
$\Frob_{\cQ} \in \Gal(K/L)$, the Frobenius element at $\cQ$. The result follows because all the elements in $\Gal(L/K) \simeq \oplus \Z/p\Z$ have order dividing $p$.
\end{proof}

\begin{lemma} \label{lem:splitting}
Let $M_1 / L$ and $M_2 / L$ be finite extensions. Let $\fp$ be a prime in $L$ that splits completely in $M_1$ and~$M_2$. Then $\fp$ splits completely in the compositum extension $M_1 M_2 / L$.
\end{lemma}
\begin{proof} This follows from \cite[Lemma 3.3.30]{cohen2}.
\end{proof}

\section{A smaller extension \texorpdfstring{$K_S$}{}}
\label{sec:smaller}
The extension $K_S$ described in Section~\ref{sec:grenie} quickly becomes computationally impractical to use even for small parameters $n$, $p$ and~$S$.
In Theorem~\ref{thm:smaller} we show that one can add an additional constraint in each layer $K_i$ which often forces the tower to stabilize early. This constraint is implicit in some of Greni\'e's calculations and is also mentioned in \cite[\S 6.2]{duan}.

We keep all the notation from Section~\ref{sec:auxiliary}. Let also $\lambda$, $\varepsilon$ and~$m$ be as in Section~\ref{sec:grenie}.

We now construct $K_S$ as follows.
Let $K_0 = K$ and, for $i \geq 0$, define $K_{i+1}$ to be the maximal $p$-elementary extension of $K_i$ unramified outside~$S$ such that $K_{i+1} / K$ is Galois and $\Gal(K_{i+1}/K)$ has exponent dividing~$p^m$. If $K_{i+1} = K_i$ for some $i < \lambda + \varepsilon + m$ then stop and take $K_S = K_i$; otherwise, we reach $i=\lambda + \varepsilon + m$ and we take $K_S = K_{\lambda + \varepsilon + m}$ as  defined in Section~\ref{sec:grenie}.

We now state the variant of Greni\'e's theorem that we implemented.
Its proof follows closely the proof of \cite[Proposition 9]{Grenie}
and~\cite[Proposition~3.5]{duan}.

\begin{theorem}\label{thm:smaller}
Let $\rho_1, \rho_2 : G_K \to \GL_n(E)$ be continuous representations unramified outside $S$ and with residual image a $p$-group.
Let $\Sigma \subset G_K$ be a set and define $\Sigma' = \{ \sigma^k : \sigma \in \Sigma, k \in \Z_{\geq 0}\}$.
Suppose that
\begin{enumerate}
 \item[(i)]  $\Sigma'$ covers the generators of each maximal cyclic subgroup of $\Gal(K_S/K)$; and
 \item[(ii)]  $\forall \sigma \in \Sigma$, $\operatorname{CharPoly}(\rho_1(\sigma)) = \operatorname{CharPoly}(\rho_2(\sigma))$.
\end{enumerate}
Then the semisimplifications of $\rho_1$ and $\rho_2$ are isomorphic.
\end{theorem}
\begin{proof} For a pro-$p$ group~$H$, let $H^\# \subset H$ be the closure of the subgroup generated by the $p^m$ powers in~$H$ and define $H_\# := H/H^\#$; see \cite[Remark 6]{Grenie} for properties of $H^\#$.

Since $G_K$ is compact there is a lattice in $E^n$ stable for the action of $G_K$, so we can assume that $\rho_i$ is valued in $\GL_n(\calO_E)$. By assumption, the image of $\rhobar_i$ is a $p$-subgroup of $\GL_n(\calO_E/\fp_E)$ whose $p$-Sylow subgroups are conjugated to the subgroup $\bar{U}$ of unipotent matrices. So, after replacing $\rho_i$ by a conjugate, we can assume its image belongs to $U \subset \GL_n(\calO_E)$ the subgroup of matrices reducing to~$\bar{U}$. Since the kernel of reduction is pro-$p$ it follows that $U$ is pro-$p$ and
$\rho_i(G_K)$ are closed in $U$ because they are compact, hence $\rho_i(G_K)$ is also pro-$p$. Thus $\rho_i$ factor via $G = \Gal(K_S(p)/K)$ the maximal pro-$p$ extension of $K$ unramified outside $S$.

We have $G / G^{\#} \simeq \Gal(K_{S,p^m} / K)$ where $K_{S,p^m} / K$ is the largest pro-$p$ extension of $K$ unramified outside $S$ with Galois group having exponent a divisor of $p^m$.

We now divide into two cases according to the construction of $K_S$:

(i) Suppose that $K_S = K_i$ with $i < \lambda + \varepsilon + m$.

Clearly, $K_S \subset K_{S,p^m}$. The maximality of $K_{i+1} / K_i$ together with $K_{i+1}=K_i$ implies $K_S = K_{S,p^m}$;
indeed, any $p$-extension~$M/K$ is uniquely constructed as the tower $K = L_0 \subset  L_1 \subset \cdots \subset L_k = M$ where $L_{i+1}$ is the maximal $p$-elementary extension
of~$L_i$ contained in $M$.

We now recall the definition of deviation group.
Let $\rho = \rho_1 \oplus \rho_2$ and consider it as a $\calO_E$-algebra homomorphism
\[
 \rho  : \calO[G] \to M_n(\calO_E) \oplus M_n(\calO_E).
\]
Let $M$ be its image and consider the composition
\[
 \delta : G \to M^\times \to (M/\fp_E M)^\times.
\]
The image $\delta(G)$ is {\it the deviation group associated to $(\rho_1, \rho_2)$}.
It follows from the work of Faltings that the semisimplifications of $\rho_1$ and $\rho_2$ are isomorphic if and only if the traces of $\rho_1$ and~$\rho_2$ agree on a set of elements $\cT$ covering $\delta(G)$; in particular, the extension $K_\delta/K$ fixed by~$\ker \delta$ is unramified outside~$S$; see~\cite[\S5.2]{Chenevert} and \cite[\S3]{duan} for more details.

In the proof of \cite[Proposition 9]{Grenie} it is shown that if $g \in G$ satisfies that $\rho_1(g)$ and $\rho_2(g)$ have the same characteristic polynomial, then $\delta(g)$ has order dividing $p^m$.

Note that the elements of $\delta(G)_{\#}$ are of the form $\overline{\delta(g)}$ for some $g \in G$ where $\bar{}\bar{\cdot}\bar{}$ denotes reduction modulo $\delta(G)^\#$. Suppose there is a set $\cT \subset G$ of elements where
$\rho_1(g)$ and $\rho_2(g)$ have the same characteristic polynomial
covering $\delta(G)_{\#}$. Then, for all $\overline{\delta(g)} \in \delta(G)_\#$, there is $g_0 \in \cT$ such that $\overline{\delta(g)} = \overline{\delta(g_0)}$
with $\delta(g_0) \in \delta(G)[p^m]$. Therefore, by \cite[Lemma 7]{Grenie},
we conclude that $\delta(G)^\#$ is trivial; hence $\delta(G) = \delta(G)_{\#}$ is covered by~$\cT$ and thus $\rho_1$ and $\rho_2$ have isomorphic semisimplifications, as desired.

To finish case (i) we have to show that $\cT = \Sigma'$ covers $\delta(G)_\#$.

Indeed, observe that the characteristic polynomials of $\rho_1(g)$ and $\rho_2(g)$ agree for all $g \in \Sigma'$, since (by assumption) they agree for $g \in \Sigma$  and the characteristic polynomial of $\rho_i(g^k)$ is determined by that of $\rho_i(g)$.
Moreover, it is clear that $\Sigma'$ covers $\Gal(K_S / K)$.

Since $\delta(G)_\#$ has exponent dividing $p^m$ and is isomorphic to
$\Gal(K_{\delta,p^m}/K)$ where $K_{\delta,p^m}/K$ is unramified outside~$S$, it follows from the maximality of $K_S = K_{S,p^m}$ that $\Gal(K_S / K)$ surjects
onto $\Gal(K_{\delta,p^m}/K)$. Thus
$\Sigma'$ covers $\delta(G)_\#$, as desired.

(ii) Suppose $K_S = K_{\lambda+\varepsilon+m}$ is the extension in Section~\ref{sec:grenie}. Then the conclusion follows directly from an application of Theorem~\ref{thm:Grenie} with $T$ the set of primes whose Frobenius elements represent
the elements of $\Sigma'$ viewed as elements of $\Gal(K_S / K)$.
\end{proof}

\begin{remark} \label{rem:traces}
Let $\cT \subset G_K$ be a set of Frobenius elements
representing all the elements of $\Sigma$ viewed as elements of $\Gal(K_S / K)$.
As explained in \cite[Remark 4]{Grenie}, the equality of characteristic polynomials
at $\Frob_t \in \cT$ in assumption (ii) can be checked using only the traces at powers of~$\Frob_t$,  more precisely,
\[
\forall \Frob_t \in \cT \text{ and } 1 \leq k \leq n \text{ check } \Tr \rho_1(\Frob^k_t) = \Tr \rho_2(\Frob^k_t).
\]
This can be very helpful in applications where only traces are quickly accessible;
it is also often the case that we know that the representations  have the same determinant, in which case, it suffices to check the above equality of traces for $k \leq n-1$.
\end{remark}

To finish this section, we will give an equivalent description of the layers $K_{i+1} / K_i$. This is the description that we use in the implementation, following the steps in Section~\ref{sec:steps}. We will need the following auxiliary results.

Recall that $L/K$ is Galois and $S \supset S_p$ is a finite set of primes in $K$.
Since $L/K$ is Galois, for all primes $\fq$ in $K$, the residual degree $f(\fq_L | \fq)$ is independent of $\fq_L \mid \fq$ in $L$.

\begin{lemma} \label{lem:smaller}
Let $\fq \not\in S$ be a prime in~$K$ satisfying $f(\fq_L | \fq) \leq p^m$.
Then there is a maximal elementary $p$-extension~$M_\fq/L$ unramified outside~$S$ and such that  $f(\cQ | \fq) \leq p^m$ at all primes~$\cQ \mid \fq$ in~$M_\fq$. Moreover, the extension $M_\fq/K$ is Galois.
\end{lemma}
\begin{proof} We consider two cases:

(i) Suppose $f(\fq_L | \fq) = p^k < p^m$. Then, by Lemma~\ref{lem:resDegree}, we have $f(\cQ | \cQ \cap \O_K) \leq p^{k+1} \leq p^m$ for all $\cQ \mid \fq$ in any $p$-elementary extension of~$L$. In this case, we take $M_\fq$ to be the maximal $p$-elementary extension of $L$ unramified outside~$S$.

(ii) Suppose $f(\fq_L | \fq) = p^m$. Consider finitely many $p$-elementary extensions $L_i / L$  unramified outside~$S$ and such that all $\fq_L \mid \fq$ split completely in all the~$L_i$. Let $M'$ be their compositum. By Lemma~\ref{lem:splitting}, $\fq_L$ splits completely in $M'$ therefore $f(\cQ | \fq) = p^m$ for all $\cQ \mid \fq$ in $M'$.

Since the $L_i / L $ above are contained in the maximal $p$-elementary extension of $L$ unramified outside $S$, there are only finitely many possibilities for $L_i$. Let $M_\fq$ be the compositum of them all. From the above $f(\cQ | \fq) \leq p^m$ for all $\cQ \mid \fq$ in $M_\fq$ so $M_\fq/L$ is $p$-elementary and maximal with respect to this property (this also implies that $M_\fq$ is among the $L_i$). Note that, for all $\sigma \in G_K$, the ideal $\sigma(\fq_L)$ splits completely in the $p$-elementary extension $\sigma(L_i)/L$. So $\sigma(L_i) \subset M_\fq$ by maximality thus $M_\fq/K$ is Galois.
\end{proof}

\begin{corollary}\label{cor:smaller}
There is a maximal elementary $p$-extension~$M/L$ unramified outside~$S$
satisfying that  $f(\cQ | \cQ \cap \O_K) \leq p^m$ at all unramified primes~$\cQ$ in~$M$.
Moreover, $M/K$ is Galois.
\end{corollary}
\begin{proof}
Let $S'$ be a finite set of primes in $K$ such that $S' \cap S$ is empty.
For each $\fq \in S'$, we obtain an extension $M_\fq$ from Lemma~\ref{lem:smaller}.
Consider the extension $M' /L$ obtained by intersecting the $M_\fq$ for all primes in $S'$.
Finally enlarging $S'$ will make $M'$ shrink and eventually converge to the desired
$M / L$. The last statement follows because the intersection of Galois extensions is Galois.
\end{proof}

We can now give an equivalent description for $K_{i+1}/K_i$.
Indeed, let $K_0 = K$ and define by $K_{i+1}$ to be the extension obtained by applying Corollary~\ref{cor:smaller} with $L = K_i$. Since every element in $\Gal(K_{i+1}/K)$ can be thought as a Frobenius element $\Frob_{\cQ \mid \fq}$ whose order is $f(\cQ | \fq) \leq p^m$, it is clear that both descriptions give the same field $K_{i+1}$.

\section{The construction of \texorpdfstring{$K_S$}{}} \label{sec:steps}
Let $n \geq 2$ and $p$ be a prime. Let $m$ be the smallest integer such that $p^m \geq n$. Let $K$ be a number field containing the $p$-th roots of unity and $S \supset S_p$ a finite set of primes in $K$.
The algorithm to construct $K_S$ is the following.

{\bf Part A: The first $m$ layers.}
Take $K_0 = K$. For $0 \leq i \leq m-1$ we let $K_{i+1}$ be the maximal $p$-elementary extension of~$K_i$ unramified outside~$S$. Explicitly, $K_{i+1} = K_i(\sqrt[p]{\alpha_1},\dots, \sqrt[p]{\alpha_d})$ where $\alpha_i$ are generators of $\Sel_p(K_i,S)$.
From Lemma~\ref{lem:resDegree} the residual degree at unramified primes increases at most by $p$ in each layer, so the maximal residual degree in $K_m$ is $\leq p^m$; thus all unramified primes in $K_m$ respect the residual degree bound $p^m$.

{\bf Part B: The layers $K_{i+1}$ for $i \geq m$.}
We need to make sure that the residual degree bound $p^m$ is preserved at each layer (see Corollary~\ref{cor:smaller}). This requires various steps.

{\bf Step 1: Computing the $\Gal(K_i/K)$-submodules.}
Let $L / K_i$ be the maximal $p$-elementary extension unramified outside $S$. We know that $K_{i+1} \subset L$ and $K_{i+1} / K$ is Galois by Corollary~\ref{cor:smaller}. Hence, by Proposition~\ref{prop:correspondence}, there is a $\Gal(K_i/K)$-submodule $N \subset \Sel_p(K_i,S)$ corresponding to~$K_{i+1}$, hence we compute all $\Gal(K_i/K)$-submodules of $\Sel_p(K_i,S)$.

To identify the submodule $N$ we will discard all submodules that do not respect the residual degree bound~$p^m$. We choose a working bound~$B_i$.

{\bf Step 2: Eliminating the 1-dimensional submodules.}
From Lemma~\ref{lem:nminus1} there is a 1-dimensional
$\Gal(K_i/K)$-submodule of~$N$ corresponding (by Proposition~\ref{prop:correspondence})
to a subfield $K_i(\sqrt[p]{\alpha}) \subset K_{i+1}$ such that $K_i(\sqrt[p]{\alpha}) / K$ is Galois.
 For each (previously computed) 1-dimension submodule~$M = \langle \alpha_1 \rangle$ we do the following. For each prime~$q < B_i$ we consider a prime $\cQ \mid q$ in
$K_i(\sqrt[p]{\alpha_1})$ and compute the residual degree $f(\cQ \mid \cQ \cap \O_K)$. If  $f(\cQ \mid \cQ \cap \O_K) >  p^m$ then we discard~$M$ (the choice of~$\cQ$ does not matter since $K_i(\sqrt[p]{\alpha_1}) / K$ is Galois).

Note any $\cQ$ in $K_{i+1}$ satisfying $f(\cQ \mid \cQ \cap \O_K) > p^m$ also satisfies $f(\cQ \cap \O_{K_i} \mid \cQ \cap \O_K) = p^m$, because by Lemma~\ref{lem:resDegree} the residual degree increases at most by~$p$ in $p$-extensions; we use this last condition to narrow the search for~$\cQ$.

Let $\cM_1$ be the set of 1-dimensional $\Gal(K_i/K)$-submodules that were not eliminated. In particular, the extensions of~$K$ corresponding to 1-dimensional submodules not in $\cM_1$ do not respect the residual degree bound at some prime $\cQ \mid q$ where $q < B_i$ is a prime.

{\bf Step 3: Eliminating submodules of dimension $> 1$.}
Suppose there is a dimension 2 submodule $M \subset N$.
From Lemma~\ref{lem:nminus1} we know that $M$ contains a 1-dimensional submodule and
from Lemma~\ref{lem:removeExtension} it cannot contain any 1-dimensional
submodule not in~$\cM_1$. This eliminates all dimension 2 submodules except those for which all dimension 1-submodules are in $\cM_1$. For each dimension 2 submodule $M = \langle \alpha_1, \alpha_2 \rangle$ not eliminated,
we proceed as above, that is, we search for primes $\cQ \mid q$ in $K_i(\sqrt[p]{\alpha_1},\sqrt[p]{\alpha_2}) / K$ with $f(\cQ \mid \cQ \cap \O_K) >  p^m$.

We let $\cM_2$ be the set of the 2-dimensional $\Gal(K_i/K)$-submodules that were not eliminated
and repeat the above process with the 3-dimensional submodules. We continue
until we eliminate all submodules in a single dimension or we reach the maximal dimension.

{\bf Step 4: Constructing the field $K_{i+1}$.} If submodules $M_{d_1}$ and $M_{d_2}$
of dimensions~$d_1$ and~$d_2$, respectively, survived the elimination procedure in steps 2 and 3 then the submodule generated by the union of the generators of the $M_{d_i}$
has dimension $ \leq d_1 + d_2$ and will also survive the elimination procedure (apply Lemma~\ref{lem:splitting} to the corresponding extensions). Therefore, if the elimination procedure terminated because we discarded all submodules of dimension $k+1$ this means that in dimension $k$ there was only 1 submodule left to eliminate; since in the maximal dimension there is a unique submodule, the described elimination procedure always ends with a uniquely identified submodule $N$.
We take $K_{i+1} / K_i$ to be the extension corresponding to $N$.

\begin{remark} The extension $K_{i+1} / K_i$ obtained by the above algorithm can be larger than necessary if the bound $B_i$ is not large enough. In that case $K_{i+1}$ is the intersection of finitely many extensions $M_\fq / K_i$ given by Lemma~\ref{lem:smaller} but it is different from the extension~$M$ given by Corollary~\ref{cor:smaller} that we are looking for. Regardless, $K_{i+1}$ is still
contained in the field~$K_S$ defined in \S\ref{sec:grenie}, $K_{i+1}/K$ is Galois and contains~$M / K$ as a quotient. Therefore, when the algorithm stops, it produces a field that contains the $K_S$ from Theorem~\ref{thm:smaller} as a quotient, hence it can be used to produce a set~$\Sigma$ to apply Theorem~\ref{thm:smaller}.
\end{remark}

\begin{remark} \label{rem:eliminate}
To eliminate submodules of dimension $\geq 2$ we can still restrict ourselves to only work wih extensions of the form $K_i(\sqrt[p]{\alpha})$ without ever needing to compute poly-$p$-extensions which would become of impractical degree very fast.
For example, say that we want to eliminate the $\Gal(K_i / K)$-submodule
$M = \langle \alpha_1, \alpha_2 \rangle$ that contains $\langle \alpha_1 \rangle \in \cM_1$ but is not the compositum of two submodules in $\cM_1$.
Then, if the prime $\cQ \mid q$ in $K_i(\sqrt[p]{\alpha_1}, \sqrt[p]{\alpha_2})$
satisfies $f(\cQ \mid \cQ \cap \O_K) >  p^m$ this means that $f(\cQ \cap \O_{K_i(\sqrt[p]{\alpha_2})} \mid \cQ \cap \O_K)  >  p^m$. Therefore, in general, we only need to determine residual degrees of extensions of the form $K_i(\sqrt[p]{\alpha}) / K_i$. Note however these are not Galois over~$K$, otherwise they correspond to submodules of dimension 1 which have survived the elimination.
\end{remark}

\begin{remark} Using the {\tt Magma} routines {\tt Submodules} and {\tt SubmoduleLattices} is the straightforward way to enumerate and order by inclusion the $G = \Gal(K_i/K)$-submodules of the Selmer group $M = \Sel_p(K_i,S)$
required to construct the layer $K_{i+1}/K_i$. These routines implement the very fast MeatAxe algorithm, but they keep all the submodules in memory, which becomes prohibitive even for small $n$, $p$ and~$S$.
For example, for the calculations in~\S\ref{sec:exmod3} there are more than 20 million submodules which requires about 52GB of RAM; furthermore, in practice, we often only need a small subset of all submodules, and mainly of small dimensions.
To make the algorithm amenable to computers with limited memory, we implemented two functions that allow us to work inductively as follows. One function computes all the 1-dimension submodules and another computes all submodules of dimension $d$ containing a specific submodule of dimension~$d-1$. More precisely, let $N \subseteq M$ is a submodule of dimension $d-1$.
Then the submodules of $M$ of dimension $d$ containing $N$ are in bijection with the one-dimensional $G$-invariant subspaces of the quotient $M/N$.
These subspaces can be obtained by computing eigenvectors for the $G$-action in $M/N$; note that since $G$ is a $p$-group all eigenvalues are 1 as $G$ acts unipotently by the same argument as in the beginning of the proof of Theorem~\ref{thm:smaller}. This inductive procedure avoids storing the entire lattice of submodules while still recovering all the relevant submodules systematically; although it is not as fast as a full MeatAxe computation, in practice it remains extremely fast and vastly more memory-efficient; for example, the total memory required for the computations in \S\ref{sec:exmod3} is reduced to around 600MB.
\end{remark}

\section{The example of Greni\'e}
\label{sec:exmod2}

In~\cite[Proposition~3.1]{vanGeemenTop}, van Geemen and~Top conjectured an equivalence of two representations, one arising on the cohomology of a surface and the other an automorphic representation for~$\GL_{3,\Q}$; this led Greni\'e to establish Theorem~\ref{thm:Grenie} and prove their conjecture.

In this section, we apply our algorithm to this same example and, as expected, we reach the same conclusion as Greni\'e, that is, Corollary~\ref{cor:2adic}.

We run our algorithm setting the number of threads in Magma to 4. Assuming GRH, it takes around 30h to run the procedure \texttt{ExampleOfGrenie} in the main file in~\cite{git}. However, if we only want to computationally back up the claims done by Greni\'e the required time is less, as we already have polynomials defining the layers provided in~\cite{Grenie}; we also provide code for this in~\cite{git}.
Next we describe the results and summarize the calculations.

This is a 3-dimensional example with $p=2$, $K=\QQ$, $S=\{2\}$ and $m=2$ (since $p^2 \geq 3$), hence the bound on the residual degree for the extension~$K_S / K$ is $p^2 = 4$.


Since $m=2$, in Part A of the algorithm (see~\S\ref{sec:steps}), we construct two layers that are maximal with respect to the previous layer without having the residual degree bound into account.

Indeed, $K_0 = \QQ$ and $K_1 = \QQ(\sqrt{2}, i)$ is the maximal $2$-elementary extension of $\QQ$ unramified outside $\{2\}$. The the 2-Selmer group $\Sel_2(K_1,\{2\})$ is a 3-dimensional $\F_2$-vector space. The extension given by adjoining all the square roots of the 3 generators is isomorphic to:
\[x^{32} - 8x^{30} + 40x^{28} - 120x^{26} + 160x^{24} - 136x^{22} + 200x^{20} - 376x^{18} + 1442x^{16} - 904x^{14}\]\[ - 1528x^{12} + 1992x^{10} + 1056x^{8} - 456x^{6} - 24x^{4} + 8x^{2} + 1.\]
The Galois group $\Gal(K_2 / \Q)$ is isomorphic to $C_2^3\rtimes C_4$. The Galois group $\Gal(K_2 / K_1)$ is~$C_2^3$ which is isomorphic to the full $2$-Selmer group $\Sel_2(K_1,\{2\})$ as expected.

We now proceed to Part B to compute $K_i$ for $i > m=2$.

The $2$-Selmer group $\Sel_2(K_2,\{2\})$ is a 17-dimensional $\F_2$-vector space.
Following Step~1 of Part~B, we compute all the $\Gal(K_2/K)$-submodules of $\Sel_2(K_2,\{2\})$, obtaining the following number of submodules per dimension.
\[
\begin{array}{|c|c|c|c|c|c|c|c|c|c|c|c|c|c|c|c|c|c|}\hline
\text{Dimension}&1&2&3&4&5&6&7&8&9&10&11&12&13&14&15&16&17\\\hline
\text{\# of submodules}&7&19&27&35&75&107&99&99&123&91&39&31&27&7&3&1&1\\\hline
\end{array}
\]
To eliminate submodules, we choose the working bound $B_3=100$; the rational primes $q \leq B_3$ satisfying~$f(\mathcal{Q}/q) = p^m = 4$ for $\cQ \mid q$ in $K_2$ are
$\{ 3, 5, 7, 11, 13, 19, 23, 29, 37, 43, 53, 59, 61, 67, 71\}$.

In Step 2 there is only one dimension~1 submodule that survives elimination.
This submodule is contained in four dimension 2 submodules that contain no other dimension 1 submodules. We eliminate these four following the procedure in Remark~\ref{rem:eliminate}, completing Step 3. This yields the layer $K_3$, isomorphic to the field defined by the polynomial
\[\begin{aligned}
&x^{64} - 16x^{62} + 104x^{60} - 304x^{58} + 344x^{56} + 496x^{54} + 568x^{52} - 48x^{50} + 1248x^{48} + 10336x^{46} \\
&+ 11720x^{44} + 17536x^{42} + 14168x^{40} + 52608x^{38} + 27320x^{36} - 19520x^{34} - 8414x^{32} + 50224x^{30} \\
&+ 243496x^{28} - 208624x^{26} - 259016x^{24} + 244784x^{22} + 175544x^{20} - 204656x^{18} - 4384x^{16}  \\
&+ 55712x^{14} - 8248x^{12} - 9920x^{10} + 4472x^{8} - 576x^{6} - 8x^{4} + 1.
\end{aligned}\]
\begin{remark}
    This polynomial is different from the one provided by Greni\'e  in~\cite{Grenie}. We have added a file in~\cite{git} that checks that both fields are isomorphic. Moreover, we should point out that Magma produces a different polynomial for each run, that is why we find an isomorphism at runtime with the output of Magma and the number field defined by the polynomial above.
\end{remark}
\begin{remark} Observe that the layer $K_3$ as described in Section~\ref{sec:grenie}
would be a degree $2^{17}\cdot 32$ extension of~$\Q$, which is unfeasible to compute.
\end{remark}

We have $\Gal(K_3 / \Q) \simeq C_4^2\rtimes C_4$. 
The algorithm does one final step, in order to find $K_4$. There are seven 1-dimensional $\Gal(K_3/\QQ)$-submodules of $\Sel_2(K_3,\{2\})$, however we are able to eliminate all of them, proving that $K_4 = K_3$, so the tower stabilizes and $K_S = K_3$.

\begin{remark}\label{rem:Fratini}
In~\cite{Grenie}, Greni\'e's argues that $K_4=K_3$ because there is no group of order 128 with an appropriate Frattini subgroup; this works as a sanity check for our algorithm.
\end{remark}

To finish we need to produce a list of primes $\cT$ whose corresponding Frobenius elements give the set $\Sigma$ satisfying (i) in Theorem~\ref{thm:smaller}, and check that (ii) is also satisfied. Alternatively, as explained in
Remark~\ref{rem:traces}, we can obtain a second (larger) list $\cT'$ of primes that allows
to conclude using only traces instead of characteristic polynomials.

We describe the methods to compute $\cT$ and $\cT'$.

Let $g, g', h \in \Gal(K_S / \Q)$ satisfy $g' = h g h^{-1}$. Let $\tilde{h} \in G_\Q$
be a lift of $h$ and $\Frob_t \in G_\Q$ act as~$g$ on $\Gal(K_S / \Q)$. Then
$s:= \tilde{h} \cdot \Frob_t \cdot \tilde{h}^{-1}$ lifts $g'$
and the characteristic polynomials of $\rho_1(s)$
and $\rho_2(s)$ are equal if and only if those of
$\rho_1(\Frob_t)$ and $\rho_2(\Frob_t)$ are equal.

Having computed~$G = \Gal(K_S/\Q)$, one computes its maximal cyclic subgroups $C_1,\dots, C_r$. From the previous paragraph, we can assume that the list $\{ C_i \}$
does not contain any two conjugate subgroups. Finally,
for each $1\leq i \leq r$, we find a rational prime $t$ such that $\Frob(\mathfrak{t}/t)$ generates~$C_i$ for some prime $\mathfrak{t} \mid t$ in~$K_S$. It is important to note that this list might not be the same as the one given in~\cite{Grenie}, since one can choose $t$ such that $\Frob(\mathfrak{t}/t)$ generates $C_i$ or any of its $\Gal(K_S/\Q)$-conjugates. We find six maximal cyclic subgroups up to conjugation, five of order 4 and one of order 2, having generators covered by Frobenius at
\[\cT = \{5,7,11,17,23,31\}.\]
Finally, to conclude that the semisimplifications of $\rho_1$ and~$\rho_2$ are isomorphic via Theorem~\ref{thm:smaller}, we have to check the equality of characteristic polynomials of $\rho_1(\Frob_t)$ and $\rho_2(\Frob_t)$ for all~$t \in \cT$; this was already done
by van Geemen and Top in~\cite[Proposition 3.11]{vanGeemenTop}.

Alternatively, to check the equivalence of representations using only traces at a set of primes~$\cT'$, the procedure is similar enough. Once computed (up to conjugation) generators $g_1,\dots, g_r$ of every maximal cyclic subgroup of $\Gal(K_S/\Q)$,
we need to cover the elements~$g_j^k$ for all $1 \leq j \leq r$ and all $1 \leq k \leq n$ with Frobenius elements (see Remark~\ref{rem:traces}). Again,
we can find Frobenius elements covering any $\Gal(K_S/\Q)$-conjugate of~$g_j^k$, as the traces of $\rho_i$ are invariant by conjugation. This process yields the list of primes
\[\cT' = \{5,7,11,17,19,23,31,73,137,257,337\}.\]

\begin{remark}
Notice how the list is not of size $5\cdot n + 2 = 15+2 = 17$, since there are elements that are shared in the cyclic subgroups (i.e. $C_i\cap C_j$ may be non-empty).
\end{remark}

\section{Modularity of abelian surfaces via 3-adic representations}
\label{sec:exmod3}

Modularity lifting theorems in the residually reducible case are scarce and often have many hypothesis. Our implementation of Theorem~\ref{thm:smaller} can be used to circumvent this scarcity and prove modularity of particular elliptic curves or certain abelian varieties when no general modularity theorem applies. We illustrate this by studying the modularity of the abelian surfaces listed in the LMFDB with good reduction away from 3; at the time of writing there are two such surfaces having conductors $3^7$ and $3^{10}$.


The following proposition is used below to show that the residual image hypothesis of Theorem~\ref{thm:smaller} is met.
In particular, the images are conjugated to $\left(\begin{smallmatrix} 1 & * \\ 0 & 1 \end{smallmatrix} \right)$ and so all characteristic polynomials are $(t-1)^2 \pmod{3}$ (compare with hypothesis (1) of Theorem~\ref{thm:Grenie}).

\begin{proposition}\label{prop:reducible} Let $K=\Q(\sqrt{-3})$ and $\rhobar : G_K \to \GL_2(\F_3)$ be a continuous representation ramified only at the unique prime~$\fq_3 \mid 3$ in $K$. Then $\rhobar(G_K)$ is trivial or isomorphic to $C_3$.
\end{proposition}
\begin{proof}
Since the determinant of $\rhobar$ is a character valued in $\F_3^*$  it fixes an extension $M/K$ of degree dividing~2 that is unramified outside~$\fq_3$. We shall shortly see that such a quadratic extension does not exist, thus $\det~\rhobar$ is trivial.

Note that $\rhobar$ is reducible if and only if its projectivisation $\PP\rhobar : G_K \to \PSL_2(\F_3) \simeq A_4$
is reducible. We claim that the only non-trivial Galois extensions $L/K$ unramified outside~$\{\fq_3\}$ with $\Gal(L/K)$ isomorphic to a subgroup of $A_4$ are cyclic extensions of degree 3. The result follows.

We now prove the claim. First observe that the subgroups of
$A_4$ are isomorphic to $\{1\}$, $C_2$, $C_3$, $C_2 \times C_2$
or a subgroup $H$ of order 12 having a normal subgroup isomorphic to $C_2 \times C_2$.

The claim follows from the following observations:

(1) $K$ has no quadratic extensions unramified away from $\fq_3$; indeed, the ray class group of~$K$ and modulus $\fq_3^3$ is cyclic of order 3.

(2) $K$ has exactly 4 cubic extensions unramified away from $\fq_3$;
indeed, $\Sel_3(K,\{\fq_3\}) \simeq C_3 \times C_3$ has exactly 4 order 3 subgroups; since $\zeta_3 \in K$ each subgroup gives rise to a Galois extension.

(3) none of the four extensions in (2) has quadratic extensions unramified away from 3; indeed, for each cubic $L/K$ in (2) a ray class group computation as in (1) yields the conclusion.

Alternatively to (1)--(3), a quick search using LMFDB confirms that the only extensions of~$K$ of degree dividing $12$ and unramified away from $\fq_3$ have degree 3.

Finally, since $\PP\rhobar(G_K) \simeq C_3$ we have $\rhobar(G_K) \simeq C_3$ or $C_6$ but the computation in (3) shows the latter case is not possible.
\end{proof}

\begin{remark}
The previous proposition was originally told to us by John Cremona who proved it
using his algorithm for finding all extensions of a number field which are unramified outside any finite set of primes, and having Galois group in a list which includes all subgroups of $A_4$. A more general algorithm that, in particular, determines whether a Galois  representation valued in~$\GL(\F_3)$ is reducible was developed by
Cremona--Sanna and is described in Chapter 3 of \cite{mattia}.
\end{remark}

\subsection{A surface of conductor $3^7$
}
Let $K = \Q(\sqrt{-3})$ and $\fq_3$ be the unique prime in $K$ above~3.
Let $M=\Q(b)$ where $b^6 + 3=0$.
Consider the genus 2 curve
\[C / \Q \; : \; y^2 + (x^3 + 1)y = -1\]
with LMFDB label \href{https://www.lmfdb.org/Genus2Curve/Q/2187/a/6561/1}{2187.a.6561.1}.
Its  Jacobian $J=J(C) / \Q$ does not have $\GL_2$-type but it becomes of $\GL_2(K)$-type
over $K$ and it
splits as the square of the elliptic curve
\[E / M \; : \; y^2 = x^3 - \frac{g_4}{48}x - \frac{g_6}{864},\quad g_4 = -\frac{891}{32} b^{5} + \frac{2187}{32} b^{2},\quad g_6 = \frac{15309}{32} b^{3} + \frac{67797}{64}.\]
One checks that $j(E)\in K$, so there is a curve $W/K$ with $j(W) = j(E)$. More precisely, the base change to~$M$ of the curve $W/K$ with LMFDB label \href{https://www.lmfdb.org/EllipticCurve/2.0.3.1/6561.1/a/2}{6561.1-a2} is isomorphic to $E/M$.

Since $W/K$ is modular by \cite[Theorem~1.1]{CaraianiNewton} and the extensions $M / K$ and $M / \Q$ are cyclic, it follows that $J/\Q$ is modular because modularity is preserved by cyclic base change.

We will now establish modularity of $W/K$ using our algorithm, avoiding the use of the general modularity theorem of Caraiani-Newton in the above argument.
Let $f$ be the Bianchi modular form of level $81\calO_K$ and coefficient field $\Q_f = \Q$ with LMFDB label~\href{https://www.lmfdb.org/ModularForm/GL2/ImaginaryQuadratic/2.0.3.1/6561.1/a/}{6561.1-a}.
Let $\rho_{f,3}$ be its 3-adic representation and $\rhobar_{f,3}$ its residual representation.
From Proposition~\ref{prop:reducible} it follows that both
$\rhobar_{f,3}$ and $\rhobar_{W,3}$ have image a $3$-group.

We run our algorithm on the machine described in \S6 with the same number of threads. Assuming GRH, it takes around 17h of runtime to run the procedure \texttt{Example3Adic} in~\cite{git}.

Next we describe the results and summarize the calculations.

We have $K_0 = K = \QQ(\sqrt{-3})$, $n=2$, $p=3$, $S = \{\fq_3\}$ and $m = 1$, hence we only have to compute one layer in Part A of the algorithm.
Indeed, the Selmer group $\Sel_3(K_0, \{\fq_3\})$ is
2-dimensional $\F_3$-vector space and $K_1$ is the extension of $\QQ$ defined by
\[x^{18} + 6x^{15} + 15x^{12} + 20x^9 + 2202x^6 + 6x^3 + 1\]
with Galois group $\Gal(K_1 / K) \simeq C_3^2$.

We now proceed to Part B. The Selmer group $\Sel_3(K_1,\{3\})$ is a 10-dimensional $\F_3$ vector space.  We computed all its $\Gal(K_1/K)$-submodules and obtained, in particular,
13 submodules of dimension 1, and 49 submodules of dimension 2.
Using the working bound $B_2=100$, we eliminate all but one dimension 1 submodules. There are 36 dimension 2 submodules containing it and we also eliminate them all following Remark~\ref{rem:eliminate}. This yields the extension $K_2 / K_1$ with Galois group~$C_3$, $K_2/K$ has Galois group
$C_3^2\rtimes C_3$ and $K_2$ is defined by the (absolute) polynomial
\[
\begin{aligned}
&x^{54} - 9x^{53} + 27x^{52} - 9x^{51} - 162x^{50} + 621x^{49} - 927x^{48} - 2853x^{47} + 17271x^{46} - 27573x^{45}\\
& - 27216x^{44} + 199728x^{43} - 349578x^{42} - 215937x^{41} + 2463120x^{40} - 4329000x^{39} - \\
& 2196252x^{38} + 19087911x^{37} - 19754406x^{36} - 35947710x^{35} + 127700802x^{34} - \\
& 125645994x^{33} - 109373778x^{32} + 459822843x^{31} - 397050057x^{30} - 508686399x^{29} + \\
& 1533961935x^{28} - 935121814x^{27} - 1593661833x^{26} + 3194569692x^{25} - 906284115x^{24} - \\
& 3479989338x^{23} + 4441500135x^{22} + 101725164x^{21} - 4788925272x^{20} + 3255745041x^{19} + \\
&  2214881976x^{18} - 3829955346x^{17} + 659765313x^{16} + 1846125180x^{15} - 1804908330x^{14} +\\
&  380728098x^{13} + 1400962572x^{12} - 1176843159x^{11} - 720091548x^{10} + 827937882x^{9} + \\
& 322896141x^{8} - 297229302x^{7} - 125213643x^{6} + 40510512x^{5} + 26411823x^{4} + \\
& 4309533x^{3} + 496962x^{2} + 28782x + 991.  \\
\end{aligned}
\]

The next step is to compute $K_3$ from $K_2$. The Selmer group $\Sel_3(K_2,\{3\})$ is a 28-dimensional $\F_3$-vector space giving rise, in particular, to 40 dimension 1 $\Gal(K_2/K_0)$-submodule.
Using the working bound $B_3=100$ we discard them all, hence $K_3 = K_2 = K_S$.

Following the procedure described at the end of Section~\ref{sec:exmod2}, we can find a set~$\cT$ of primes in~$K$ whose corresponding Frobenius elements cover generators of all the maximal cyclic subgroups of $\Gal(K_S/K)$ up to conjugation.
Indeed, there are five such subgroups $C_1,\dots, C_5$ all of order~3,
hence, each of the two non-trivial elements in~$C_i$ generates $C_i$.
We found that there are Frobenius elements in~$K$ above the primes
$\{ 2, 5, 7, 19, 73 \}$ covering (conjugates of) the non-trivial elements of the $C_i$ in the following manner.

Note that $2$ and~$5$ are inert in $K$. We check that the primes in~$K$ above~73 and~19 cover the non-trivial elements in $C_1$ and~$C_4$, respectively; the primes above~7 cover a non-trivial element in $C_2$
and one in~$C_3$ whilst the primes above~13 cover the other non-trivial elements in these groups;
finally, the primes $2\O_K$ and $5\O_K$ cover the non-trivial elements in~$C_5$.

For a rational prime~$q$ splitting in~$K$, let $\fp_{q}^1$, $\fp_{q}^2$ denote the two primes in $K$ above it and~$\fq_q^\ast$ denote any of them.
From the above, we conclude that the list of primes
\[\cT = \{2\calO_K, \mathfrak{p}_7^1, \mathfrak{p}_7^2, \mathfrak{p}_{19}^{\ast}, \mathfrak{p}_{73}^{\ast}\} \quad \text{ or } \quad \cT = \{5\calO_K, \mathfrak{p}_7^1, \mathfrak{p}_7^2, \mathfrak{p}_{19}^{\ast}, \mathfrak{p}_{73}^{\ast}\}
\]
suffices to cover one generator (up to conjugation) of each of the $C_i$.
Therefore, under GRH, we can now easily derive the following consequence;
note that Corollary~\ref{cor:3adic} is its part 2).

\begin{corollary}\label{cor:mod3}
Let $\rho_1,\rho_2 : G_K \to \GL_2(\Z_3)$ be continuous and
unramified~outside~$\fq_3$.

1) Then $\rho_1$ and~$\rho_2$ have isomorphic semisimplifications if and only if $\rho_1(\Frob_t)$ and $\rho_2(\Frob_t)$ have the same characteristic polynomials for all $t \in \cT$.

2) Suppose that $\det~\rho_1 = \det~\rho_2$.
Then $\rho_1$ and~$\rho_2$ have isomorphic semisimplifications if and only if $\rho_1(\Frob_t)$ and $\rho_2(\Frob_t)$ have the same traces for all $t \in \cT$.
\end{corollary}\label{cor:general3}
\begin{proof} From Proposition~\ref{prop:reducible} we know that both residual images are a 3-group.

Part 1) follows directly from Theorem~\ref{thm:smaller} and the calculation of~$K_S$ and~$\cT$ described above.

For part 2) observe from Remark~\ref{rem:traces} that we have to compare traces at Frobenius elements covering~$g_i^k$ for $1 \leq i \leq 5$ and~$1\leq k \leq n-1 =1$
where~$\{g_i\}$ is any choice of generator of the~$C_i$;
from 1) we already know this is achieved by $\cT$ and the result follows from Theorem~\ref{thm:smaller}.
\end{proof}

We can now quickly finish the proof of modularity of~$W$. Indeed,
the representations $\rho_{W,3}$ and $\rho_{f,3}$ are 2-dimensional and both have determinant the $3$-adic cyclotomic character which is trivial over~$K$;
therefore, modularity of~$W$ follows from Corollary~\ref{cor:general3} part 2)
if we check equality of traces
at $\Frob_t$ for all $t \in \cT$; this can be checked easily
using the LMFDB:
\[
\begin{array}{|l|c|c|c|c|c|}
    \hline
     & 2\O_K & \mathfrak{p}_7^1 & \mathfrak{p}_7^2 & \mathfrak{p}_{19}^{\ast} & \mathfrak{p}_{73}^{\ast}    \\ \hline
    \Tr (\Frob_t(\rho_{W,3})) & -1 & -1  & -1 & -1 & 2 \\ \hline
    \Tr (\Frob_t(\rho_{f,3})) & -1 & -1 & -1 & -1 &  2 \\ \hline
\end{array}
\]

\subsection{A surface of conductor $3^{10}$}
Consider the genus 2 curve
\[C / \Q \; : \; y^2 + (x^3+1)y = x^3-1\]
with LMFDB label \href{https://www.lmfdb.org/Genus2Curve/Q/59049/a/177147/1}{59049.a.177147.1}. Recall that $K=\Q(\sqrt{-3})$ and $M=\Q(b)$ where $b^6 + 3=0$.
The Jacobian $J=J(C) / \Q$ does not have $\GL_2$-type but it becomes of $\GL_2(K)$-type
over $K$ and it
splits over~$M$ as the square of the elliptic curve
\[
E / M \; : \; y^2 = x^3 - g_4/48x - g_6/864,\ \text{ with }
\]
\[g_4 = -\frac{2846016}{28561}b^5 + \frac{2686608}{28561}b^2,\quad g_6 = - \frac{1158095232}{4826809}b^3 + \frac{22752498240}{4826809}.\]

One checks that $j(E)\in K$, so there is a curve $W/K$ with $j(W) = j(E)$. More     precisely, the base change to~$M$ of the curve $W/K$ with LMFDB label \href{https://www.lmfdb.org/EllipticCurve/2.0.3.1/729.1/a/1}{729.1-a1} is isomorphic to $E/M$.

As in the previous example, the modularity of $W/K$ follows from the work of Caraiani--Newton and modularity of $J/\Q$ by cyclic base change. Alternatively, we can use Corollary~\ref{cor:3adic} to prove that $W/K$ is modular. Indeed, let $f$ be the Bianchi modular form of level $27\O_K$ and coefficient field $\Q_f = \Q$ with LMFDB label \href{https://www.lmfdb.org/ModularForm/GL2/ImaginaryQuadratic/2.0.3.1/729.1/a/}{729.1-a}. Since the determinants of $\rho_{f,3}$ and $\rho_{W,3}$ are the 3-adic cyclotomic character, we deduce modularity from Corollary~\ref{cor:3adic} as in the previous example; the equality of traces can also be checked using the LMFDB.

\section{Avoiding GRH}

For the calculations in \S\ref{sec:exmod2}-\S\ref{sec:exmod3} to finish in a reasonable amount of time, we had to assume GRH. In particular, to complete the proof of Corollary~\ref{cor:mod3} (and hence of Corollary~\ref{cor:3adic}), we have to avoid~GRH in the computation of $\Sel_3(K_2,\{3\})$.
We achieve this as a consequenve of the following stopping creterion for the algorithm described in Section~\ref{sec:steps}.

\begin{theorem}\label{thm:stopping}
Let $n \geq 2$ and $p \geq n$ be a prime. Let $K$ be a number field containing the $p$-th roots of unity and $S \supset S_p$ a finite set of primes in $K$. Assume the algorithm in Section~\ref{sec:steps} reaches a layer $K_{i+1} / K_i$ such that $\Gal(K_{i+1} / K_i) = \Z/p\Z$.
Then $K_S = K_{i+1}$ is the final layer.
\end{theorem}
\begin{proof} Since $p \geq n$, we have $m=1$ when applying the algorithm in \S\ref{sec:steps}. For a contradiction, suppose that $K_S \neq K_{i+1}$. Then $K_{i+2} / K_{i+1}$ is a non-trivial extension and, by Lemma~\ref{lem:nminus1} and Proposition~\ref{prop:correspondence}, there is an extension $L/K_{i+1}$ satisfying that $\Gal(L/K_{i+1}) \simeq \Z/p\Z$ and $L/K$ is Galois. In particular, $L/K_i$ is also Galois and
\[
\Gal(L/K_i) \simeq \Gal(L/K_{i+1}) \rtimes \Gal(K_{i+1}/K_i) \simeq \Z/p\Z \rtimes \Z/p\Z
\]
is a group of order $p^2$, hence $\Gal(L/K_i) \simeq (\Z/p\Z)^2$ or $\Z/p^2\Z$. We claim that neither case can occur, yielding a contradiction, and thus $K_S = K_{i+1}$.

We now prove the claim.
By the maximality of~$K_{i+1} / K_i$, the first case cannot occur (otherwise $L \subset K_{i+1}$). If
$\Gal(L/K_i)$ is cyclic of order $p^2$, then any Frobenius element in $\Gal(K_{i+2}/ K)$ that restricts to a generator of $\Gal(L/K_i)$ has order at least $p^2 > p^m = p^1$, hence it does not respect the residue degree bound, a contradiction.
\end{proof}

From the calculations in \S\ref{sec:exmod3}, we have $K=\Q(\sqrt{-3})$, $p=3$, $m=1$
and $\Gal(K_2 / K_1) \simeq \Z/3\Z$, and this can easily be checked without assuming GRH in about 1 minute.
Thus $K_S = K_2$ by Theorem~\ref{thm:stopping}, completing the proof of Corollary~\ref{cor:mod3}.


\begin{thebibliography}{10}


\bibitem{magma}
Wieb Bosma, John Cannon, and Catherine Playoust.
\newblock The {M}agma algebra system. {I}. {T}he user language.
\newblock {\em J. Symbolic Comput.}, 24(3-4):235--265, 1997.
\newblock Computational algebra and number theory (London, 1993).


\bibitem{paramodularity}
Armand Brumer, Ariel Pacetti, Cris Poor, Gonzalo Tornaría, John Voight, and
  David Yuen.
\newblock On the paramodularity of typical abelian surfaces.
\newblock {\em Algebra and Number Theory}, 13(5):1145–1195, Jul 2019.


\bibitem{CaraianiNewton}
Ana Caraiani and James Newton.
\newblock On the modularity of elliptic curves over imaginary quadratic fields. Preprint, 2025, available on \href{https://arxiv.org/abs/2301.10509}{arxiv}.


\bibitem{Chenevert}
Gabriel Ch\^enevert.
\newblock Exponential sums, hypersurfaces with many symmetries and Galois representations.
\newblock PhD thesis, 2008, McGill University, available \href{
https://www.math.mcgill.ca/goren/Students/Chenevert.thesis.pdf}{online}.


\bibitem{cohen}
Henri Cohen.
\newblock {\em Advanced topics in computational number theory.}
\newblock Grad. Texts in Math., 193, Springer-Verlag, New York, 2000. xvi+578 pp.


\bibitem{cohen2}
Henri Cohen.
\newblock {\em Number theory. {V}ol. {I}. {T}ools and {D}iophantine equations},
\newblock Grad. Texts in Math., 239, Springer, New York, 2007. xxiv+650 pp.


\bibitem{DGP}
Luis Dieulefait, Lucio Guerberoff, and Ariel Pacetti.
\newblock Proving modularity for a given elliptic curve over an imaginary
  quadratic field.
\newblock {\em Math. Comp.}, 79(270):1145--1170, 2010.


\bibitem{duan}
Lian Duan.
\newblock Faltings-{S}erre method on three dimensional selfdual
  representations.
\newblock {\em Math. Comp.}, 90(328):931--951, 2021.


\bibitem{Faltings}
G.~Faltings.
\newblock Endlichkeitss\"{a}tze f\"{u}r abelsche {V}ariet\"{a}ten \"{u}ber
  {Z}ahlk\"{o}rpern.
\newblock {\em Invent. Math.}, 73(3):349--366, 1983.


\bibitem{git}
Nuno Freitas and Ignasi Sánchez-Rodríguez.
\newblock {Supporting code files for this paper},
\newblock available on \href{https://github.com/IgnasiSanchez/ComparingReducibleReps}{github}.


\bibitem{vanGeemenTop}
Bert van Geemen and Jaap Top.
\newblock A non-selfdual automorphic representation of $\GL_3$ and a galois
  representation.
\newblock {\em Inventiones mathematicae}, 117(3):391--402, 1994.


\bibitem{Grenie}
Loïc Grenié.
\newblock Comparison of semi-simplifications of galois representations.
\newblock {\em Journal of Algebra}, 316(2):608--618, 2007.
\newblock Computational Algebra.




\bibitem{livne}
Ron Livn\'{e}.
\newblock Cubic exponential sums and {G}alois representations.
\newblock In {\em Current trends in arithmetical algebraic geometry ({A}rcata,
  {C}alif., 1985)}, volume~67 of {\em Contemp. Math.}, pages 247--261. Amer.
  Math. Soc., Providence, RI, 1987.




\bibitem{LMFDB}
The {LMFDB Collaboration}.
\newblock The {L}-functions and modular forms database.
\newblock \url{https://www.lmfdb.org}, 2025, [Online; accessed 17 July 2025]




\bibitem{mattia}
Mattia Sanna.
\newblock On the equivalence of $3$-adic galois representations.
\newblock PhD thesis, 2020, University of Warwick.
\newblock available \href{https://wrap.warwick.ac.uk/id/eprint/156281/1/WRAP_Theses_Sanna_2020.pdf}{online}.


\bibitem{Serre}
Jean-Pierre Serre.
\newblock {\em Linear representations of finite groups.} Translated from the second French edition by Leonard L. Scott,
\newblock Grad. Texts in Math., vol 42, Springer-Verlag, New York-Heidelberg, 1977. x+170 pp.






\bibitem{visser}
Robin Visser.
\newblock The Effective Shafarevich Conjecture.
\newblock Ph.D Thesis, 2025, University of Warwick.
\newblock \href{https://warwick.ac.uk/fac/sci/maths/people/staff/visser/thesis.pdf}{available online}.




\end{thebibliography}
\end{document}